\documentclass[11pt]{article}
\usepackage[a4paper,margin=2.55cm]{geometry}
\usepackage[T1]{fontenc}
\usepackage[utf8]{inputenc}
\usepackage{lmodern}
\usepackage{microtype}
\usepackage{amsmath,amssymb,amsthm,mathtools}
\usepackage{booktabs,array,tabularx}
\usepackage{enumitem}
\usepackage{xcolor}
\usepackage{hyperref}
\usepackage[nameinlink,capitalise,noabbrev]{cleveref}
\usepackage[none]{hyphenat}
\usepackage{tikz}
\usetikzlibrary{arrows.meta,positioning,calc,decorations.pathreplacing}
\usepackage{float}
\usepackage{algorithm}
\usepackage{algpseudocode}
\usepackage{orcidlink}
\hypersetup{colorlinks=true,linkcolor=blue!55!black,citecolor=blue!55!black,urlcolor=blue!55!black}
\setlist[itemize]{leftmargin=1.5em}
\setlist[enumerate]{leftmargin=1.7em}
\newcolumntype{Y}{>{\raggedright\arraybackslash}X}

\newtheorem{theorem}{Theorem}[section]
\newtheorem{lemma}[theorem]{Lemma}
\newtheorem{proposition}[theorem]{Proposition}
\newtheorem{corollary}[theorem]{Corollary}
\newtheorem{remark}[theorem]{Remark}

\theoremstyle{definition}
\newtheorem{definition}[theorem]{Definition}
\newtheorem{example}[theorem]{Example}
\newtheorem*{theoremA}{Theorem A}
\newtheorem*{theoremB}{Theorem B}
\newtheorem*{theoremC}{Theorem C}
\newtheorem*{theoremD}{Theorem D}

\newcommand{\diam}{\operatorname{diam}}
\newcommand{\diamV}{\operatorname{diam}_V}
\newcommand{\Rips}{\delta_{\triangle}}
\newcommand{\four}{\delta_{4}}

\newcommand{\N}{\mathbb{N}}
\newcommand{\R}{\mathbb{R}}

\newcommand{\cP}{\mathcal{P}}

\title{\textbf{Sharp Gromov Hyperbolicity Constants for Cartesian Products of Finitely Many Trees}}

\begin{document}
\author{{Junior Michel \orcidlink{0009-0005-8450-2757}}$^{1}$,
{Anier Soria-Lorente \orcidlink{0000-0003-3488-3094}}$^{1}$,
{Jean-Marie Vilaire \orcidlink{0000-0002-4012-6471}}$^{2}$ \\
%EndAName
\\
$^{1}$Department of Quantitative Methods, Loyola University,\\
Avda. de las Universidades, 2, Dos Hermanas, Seville,\\
41704, Andalusia, Spain\\
jmichel@uloyola.es,
asoria@uloyola.es\\
$^{2}$Mobile Computing and Networking Research Laboratory (LARIM),\\
Department of Computer and Software Engineering,\\
Polytechnique Montr\'eal, Montr\'eal, QC H3T 1J4, Canada\\
jean-marie.vilaire@polymtl.ca\\
}
\maketitle

\begin{abstract}
Let $X=T_1\times\cdots\times T_k$, $k\ge2$, be a Cartesian product of finite
nontrivial trees with unit edges, and write $d_i=\diam T_i$, $S=\sum_id_i$ and
$M=\max_id_i$. We prove that the thin-triangle hyperbolicity constant of the
geodesic realization of $X$ is
\[
 \Rips(X)=\min\Bigl\{\frac S2,\ S-M+\frac12\Bigr\},
\]
and that the four-point hyperbolicity constant of its vertex metric is the value
\[
\mathcal{P}(d_1,\dots,d_k)
=
\max_{E\subseteq\{1,\dots,k\}}
\min\left\{
\sum_{i\in E} d_i,\,
\sum_{i\notin E} d_i
\right\}.
\]
of the best two-way partition of the diameters. Both constants therefore depend only on the
diameter vector, and their difference is governed by the partition imbalance $\eta$
of that vector: it equals $\eta/2$ when $M\le S-M$ and $1/2$ otherwise. In
particular the two constants coincide if and only if the diameters admit a perfect
partition, and their difference is unbounded. For $k\ge3$ the first formula exhibits
a many-factor balancing phenomenon: several short factors can collectively balance
the longest one and force the universal bound $\frac12\diam X$ to be attained. The
proofs are self-contained and rest on an exact distance formula for the geodesic
realization of $X$. For the $\ell_1$-product of the factors, i.e.\ the
$\ell_1$-metric on the associated cube complex, we obtain
$\Rips=\min\{S/2,S-M\}$, which identifies the half-unit above as a pure
$1$-skeleton effect, and we show that the four-point formula holds for
$\ell_1$-products of arbitrary $0$-hyperbolic metric spaces, in particular for
weighted trees and $\R$-trees. We also treat products with unbounded factors and
record algorithmic consequences.
\end{abstract}

\maketitle

\section{Introduction}

Gromov hyperbolicity measures how closely a geodesic metric space resembles a tree
\cite{GhysHarpe1990,Gromov1987}; it is a basic tool in geometric group theory, in
metric graph theory and in the study of large networks
\cite{ChepoiDraganEstellonHabibVaxes2008,HernandezReynaRosario2020}. For graphs two
conventions are standard. One regards every edge as a copy of $[0,1]$ and measures
the thinness of geodesic triangles in the resulting geodesic metric graph; the other
applies Gromov's four-point condition to the shortest-path metric on the vertex
set. The two constants are quantitatively related, but their sharp values need not
coincide, and computing either of them exactly is difficult even for simple graph
families.

Cartesian products of trees are a natural testing ground. Trees are
$0$-hyperbolic, while their products are median graphs, the $1$-skeleta of
CAT(0) cube complexes; see
\cite{BandeltBuroschLaborde1996,ChalopinChepoiHiraiOsajda2020,Genevois2019}. For the
thin-triangle convention, Michel, Rodr\'iguez, Sigarreta and Villeta
\cite{MichelRodriguezSigarretaVilleta2010} characterized the hyperbolicity of a
Cartesian product of two graphs and computed the constant of a product of two trees;
further bounds and the role of bigons in median graphs are studied in
\cite{RodriguezSigarreta2012,Sigarreta2013}. For the four-point convention, Chepoi,
Dragan, Estellon, Habib and Vax\`es \cite{ChepoiDraganEstellonHabibVaxes2008} related
the hyperbolicity of median graphs to isometrically embedded square grids, and
Coudert and Ducoffe \cite{CoudertDucoffe2016} computed it for products of paths.
Algorithmic aspects are studied in
\cite{ChalopinChepoiDraganDucoffeMohammedVaxes2021,CohenCoudertLancin2015,FournierIsmailVigneron2015}.

None of these results gives the sharp constants of a product of $k\ge3$ arbitrary
trees, nor compares the two conventions on this class. This is what we do here.

\subsection*{Main results}
Throughout, $T_1,\dots,T_k$ are nontrivial trees with unit edges and
\[
 X=T_1\times\cdots\times T_k,\quad d_i=\diam T_i,\quad S=\sum_{i=1}^kd_i,\quad
 M=\max_id_i,\quad R=S-M .
\]
$\Rips$ is the sharp thin-triangle constant of the geodesic realization and $\four$
the four-point constant (\cref{sec:prelim}); for positive numbers $d_1,\dots,d_k$ we
put
\[
 \cP(d_1,\dots,d_k)=\max_{E\subseteq\{1,\dots,k\}}
 \min\Bigl\{\sum_{i\in E}d_i,\ \sum_{i\notin E}d_i\Bigr\},\quad
 \eta(d_1,\dots,d_k)=\min_{E}\Bigl|\sum_{i\in E}d_i-\sum_{i\notin E}d_i\Bigr|,
\]
so that $\cP=(S-\eta)/2$ is the largest subset sum not exceeding $S/2$.

\begin{theoremA}[\cref{thm:mainrips}]
If the trees are finite and $k\ge2$, then
\[
 \Rips(X)=\min\Bigl\{\frac S2,\ S-M+\frac12\Bigr\}
 =\begin{cases} S/2, & M\le R,\\ R+1/2, & M>R,\end{cases}
\]
and the value is attained by an explicit geodesic bigon.
\end{theoremA}

For $k=2$ this agrees with the two-tree formula of
\cite{MichelRodriguezSigarretaVilleta2010}. The new phenomenon occurs for $k\ge3$ in
the \emph{balanced regime} $M\le R$: the universal bound
$\Rips(X)\le\frac12\diam X$ is attained by a bigon whose sides traverse diametral
geodesics of the factors in opposite coordinate orders, so that several short
factors collectively balance the longest one.

\begin{theoremB}[\cref{thm:fourpoint,cor:fourtrees}]
If the trees are finite, then $\four(V(X))=\cP(d_1,\dots,d_k)$. More generally, the
four-point constant of the $\ell_1$-product of metric spaces $Y_1,\dots,Y_k$ with
$\four(Y_i)=0$ and finite diameters $D_i$ is $\cP(D_1,\dots,D_k)$.
\end{theoremB}

For products of paths this is the grid formula of \cite{CoudertDucoffe2016}. In
particular the branching of the factors is invisible to \emph{both} sharp constants
(\cref{cor:diamonly}), and Theorem~B applies verbatim to weighted trees and to
$\R$-trees.

\begin{theoremC}[\cref{thm:exactgap}]
If the trees are finite and $k\ge2$, then
\[
 \Rips(X)-\four(V(X))=\begin{cases}\eta/2, & M\le R,\\ 1/2, & M>R.\end{cases}
\]
Hence $\Rips(X)=\four(V(X))$ if and only if $\eta=0$, and for every odd $k\ge3$ the
difference is unbounded.
\end{theoremC}

The half-unit in Theorems A and C comes from the edges of $X$. To make this precise,
let $|X|_1$ denote the $\ell_1$-product of the geodesic realizations of the
factors; it is the cube complex of $X$ (all cubes filled in) with its
$\ell_1$-metric, and $X$ is its $1$-skeleton.

\begin{theoremD}[\cref{thm:l1}]
If the trees are finite and $k\ge2$, then
\[
 \Rips(|X|_1)=\min\Bigl\{\frac S2,\ S-M\Bigr\},\quad \four(|X|_1)=\cP(d_1,\dots,d_k).
\]
The same holds for $\ell_1$-products of compact $\R$-trees of diameters $d_i$.
\end{theoremD}

Thus $\Rips(X)-\Rips(|X|_1)$ is $0$ in the balanced regime and exactly $\frac12$ in
the dominant regime, and in the latter regime the two conventions coincide on
$|X|_1$.

All the proofs are self-contained. The key tool is an exact distance formula for
the geodesic realization of $X$ (\cref{lem:distance}), from which we derive the
behaviour of the coordinate projections of geodesics (\cref{lem:geodesics}) and the
upper bound $\Rips(X)\le S-d_j+\frac12$ for every $j$ (\cref{prop:upper}). We also
treat locally finite products with unbounded factors (\cref{thm:unbounded}) and
record algorithmic consequences (\cref{sec:algorithms}).

\subsection*{Relation to previous work}
We do not claim novelty for the two-tree Rips formula of
\cite{MichelRodriguezSigarretaVilleta2010} or for the path-grid four-point formula
of \cite{CoudertDucoffe2016}, which are special cases of Theorems A and B. The
median-graph literature provides qualitative characterizations of hyperbolicity in
terms of thin bigons and bounded square grids
\cite{ChalopinChepoiHiraiOsajda2020,ChepoiDraganEstellonHabibVaxes2008,Sigarreta2013}.
To the best of our knowledge, the sharp $k$-factor Rips formula, the partition
formula for arbitrary tree factors, the exact comparison of the two conventions and
the $\ell_1$ analogue have not been stated before.

\section{Preliminaries}\label{sec:prelim}

Graphs are connected and simple, with edges of length $1$ unless stated otherwise.
Each graph $G$ is regarded as a geodesic metric graph: every edge is identified with
a copy of the unit interval, so that points in the interior of edges belong to the
space, and $d_G$ is the induced length metric. We write $V(G)$ for the vertex set,
$\diam G=\sup\{d_G(x,y):x,y\in G\}$ and $\diamV G=\sup\{d_G(u,v):u,v\in V(G)\}$;
for a tree these agree. A graph is \emph{nontrivial} if it has at least two vertices.

\begin{definition}[Thin-triangle hyperbolicity]
Let $Y$ be a geodesic metric space. A \emph{geodesic triangle} is the union of three
geodesics joining three points; it is \emph{$\delta$-thin} if each side is contained
in the closed $\delta$-neighbourhood of the union of the other two sides. The sharp
\emph{Rips constant} is
$\Rips(Y)=\inf\{\delta\ge0:\text{every geodesic triangle in }Y\text{ is }
\delta\text{-thin}\}\in[0,\infty]$. A \emph{geodesic bigon} is a degenerate geodesic
triangle formed by two geodesics with common endpoints.
\end{definition}

This is the convention of
\cite{MichelRodriguezSigarretaVilleta2010,RodriguezSigarretaVilaireVilleta2011}.
Since every point of a side of length $\ell$ is within $\ell/2$ of an endpoint, and
the endpoints belong to the other sides,
\begin{equation}\label{eq:diambound}
 \Rips(Y)\le\tfrac12\diam Y
\end{equation}
for every geodesic metric space $Y$ (see \cite{RodriguezSigarretaVilaireVilleta2011}
for graphs).

\begin{definition}[Four-point hyperbolicity]
For a metric space $(Y,d)$ and $x,y,z,w\in Y$, consider the three sums
\[
 \sigma_1=d(x,y)+d(z,w),\quad \sigma_2=d(x,z)+d(y,w),\quad
 \sigma_3=d(x,w)+d(y,z).
\]
The \emph{four-point defect} $\four(x,y,z,w)$ is half the difference between the
largest and the second largest of them, and
$\four(Y)=\sup_{x,y,z,w\in Y}\four(x,y,z,w)$. For a graph $G$, $\four(V(G))$ is the
four-point constant of the vertex metric.
\end{definition}

A metric space has $\four=0$ if and only if it embeds isometrically in an $\R$-tree;
in that case the two largest of the three sums coincide for every quadruple.

The Cartesian product $G_1\times\cdots\times G_k$ has vertex set
$V(G_1)\times\cdots\times V(G_k)$, two vertices being adjacent when they differ in
exactly one coordinate and their entries in that coordinate are adjacent
\cite{HammackImrichKlavzar2011,Sabidussi1960}. Its geodesic realization is its
$1$-skeleton: every point $x$ has coordinates $\pi_i(x)\in G_i$, at most one of
which lies in the interior of an edge. An edge of $X=T_1\times\cdots\times T_k$ along
which the $i$-th coordinate changes is said to lie \emph{over} the corresponding
edge of $T_i$.

\section{The metric of a product of trees}\label{sec:lemmas}

In this section $X=T_1\times\cdots\times T_k$ is a product of trees.

\begin{definition}
Two points $x,y\in X$ form an \emph{exceptional pair in the coordinate $m$} if they
lie in the interiors of two distinct edges of $X$ over the same edge $e$ of $T_m$. In
that case we write $s,t\in(0,1)$ for the distances of $\pi_m(x)$ and $\pi_m(y)$ from
a fixed endpoint of $e$.
\end{definition}

\begin{lemma}[Distance formula]\label{lem:distance}
For all $x,y\in X$,
\begin{equation}\label{eq:distance}
 d_X(x,y)=\sum_{i=1}^kd_{T_i}(\pi_i(x),\pi_i(y))+c(x,y),
\end{equation}
where $c(x,y)=0$ unless $x,y$ form an exceptional pair in some coordinate, in which
case
\[
 c(x,y)=\min\{s+t,\,2-s-t\}-|s-t|=2\min\{s,t,1-s,1-t\}\in(0,1].
\]
Moreover, in the exceptional case every path from $x$ to $y$ moves the $m$-th
coordinate by at least $\min\{s+t,2-s-t\}$.
\end{lemma}

\begin{proof}
\emph{Lower bound.} Let $\gamma$ be a rectifiable path from $x$ to $y$. Every subarc
of $\gamma$ contained in an edge changes exactly one coordinate, so
$L(\gamma)=\sum_iL_i(\gamma)$, where $L_i(\gamma)$ is the length of the path
$\pi_i\circ\gamma$ in $T_i$; hence $L_i(\gamma)\ge d_{T_i}(\pi_i(x),\pi_i(y))$. In the
exceptional case, $\gamma$ must leave the edge containing $x$ and enter the edge
containing $y$ through vertices of $X$, whose $m$-th coordinates are endpoints of $e$;
thus $\pi_m\circ\gamma$ goes from $\pi_m(x)$ to an endpoint of $e$ and later from an
endpoint of $e$ to $\pi_m(y)$, and $L_m(\gamma)\ge\min\{s+t,2-s-t\}$. This proves the
inequality $\ge$ in \eqref{eq:distance} and the last assertion.

\emph{Upper bound.} Let $x\in[u,u']$ and $y\in[v,v']$, where $u,u'$ differ only in
the coordinate $i$ and $v,v'$ only in the coordinate $j$ (a vertex lies on any edge
containing it). For $u^\circ\in\{u,u'\}$ and $v^\circ\in\{v,v'\}$,
\begin{equation}\label{eq:routes}
 d_X(x,y)\le d_X(x,u^\circ)+\sum_{l}d_{T_l}(u^\circ_l,v^\circ_l)+d_X(v^\circ,y).
\end{equation}
If $i\ne j$, the coordinates $l\ne i,j$ contribute $d_{T_l}(\pi_l(x),\pi_l(y))$ for
every choice. In the coordinate $i$, $v_i=v'_i$ is a vertex, and the geodesic of $T_i$
from the point $\pi_i(x)$ of the edge $u_iu'_i$ to $v_i$ leaves that edge through one
of its endpoints; choosing $u^\circ$ accordingly, the terms involving the coordinate
$i$ add up to $d_{T_i}(\pi_i(x),\pi_i(y))$. The coordinate $j$ is treated in the same
way, independently, and \eqref{eq:distance} follows with $c=0$. If $i=j$ and the
edges $u_iu'_i$, $v_iv'_i$ of $T_i$ are distinct, the same argument applied to the
geodesic of $T_i$ from $\pi_i(x)$ to $\pi_i(y)$ gives $c=0$. If $i=j$ and these edges
coincide, then either $x,y$ lie on the same edge of $X$ and $d_X(x,y)=|s-t|$, or
$x,y$ form an exceptional pair; choosing $u^\circ,v^\circ$ with the same $i$-th
coordinate gives the bound $\min\{s+t,2-s-t\}+\sum_{l\ne i}d_{T_l}(\pi_l(x),\pi_l(y))$.
\end{proof}

\begin{corollary}[Diameter]\label{cor:diameter}
If the trees are finite and nontrivial, then $\diamV X=\diam X=S$.
\end{corollary}

\begin{proof}
By \eqref{eq:distance} the vertex diameter is $S$, attained by diametral pairs in the
factors. For arbitrary points, the sum in \eqref{eq:distance} is at most $S$, and in
the exceptional case the coordinate $m$ contributes $|s-t|+c=\min\{s+t,2-s-t\}\le1\le
d_m$; hence $d_X(x,y)\le S$.
\end{proof}

\Cref{cor:diameter} also follows by induction from the identity
$\diam(T\times G)=\diam T+\diam G$, valid for every tree $T$ and every graph $G$
\cite[Corollary~10]{MichelRodriguezSigarretaVilleta2010}.

\begin{lemma}[Projections of geodesics]\label{lem:geodesics}
Let $\gamma$ be a geodesic of $X$ from $x$ to $y$ and let $i\in\{1,\dots,k\}$.
\begin{enumerate}[label=\textup{(\roman*)},leftmargin=*]
\item If $x,y$ do not form an exceptional pair in the coordinate $i$, then
$\pi_i\circ\gamma$ has length $d_{T_i}(\pi_i(x),\pi_i(y))$; hence it traverses the
geodesic $[\pi_i(x),\pi_i(y)]$ of $T_i$ monotonically, and
$\pi_i(\gamma)=[\pi_i(x),\pi_i(y)]$.
\item If $x,y$ form an exceptional pair in the coordinate $i$, over the edge $e$ of
$T_i$, then $\pi_i(\gamma)\subseteq e$.
\end{enumerate}
\end{lemma}

\begin{proof}
Write $L_l=L_l(\gamma)$ and $d_l=d_{T_l}(\pi_l(x),\pi_l(y))$. By
\cref{lem:distance}, $\sum_l L_l=L(\gamma)=\sum_ld_l+c(x,y)$ and $L_l\ge d_l$ for all
$l$. If $c(x,y)=0$, then $L_l=d_l$ for all $l$. If $x,y$ form an exceptional pair in
a coordinate $m$, the last assertion of \cref{lem:distance} gives
$L_m\ge d_m+c(x,y)$, hence $L_m=d_m+c(x,y)=\min\{s+t,2-s-t\}$ and $L_l=d_l$ for
$l\ne m$. This proves (i), since a path in a tree whose length equals the distance
between its endpoints is a monotone parametrization of the geodesic joining them. In
case (ii), $\pi_i\circ\gamma$ has length $\min\{s+t,2-s-t\}$ and visits an endpoint of
$e$; it cannot leave $e$ without exceeding this length.
\end{proof}

\begin{lemma}[Sequential geodesics]\label{lem:sequential}
Let the trees be finite and, for each $i$, let $\alpha_i$ be a diametral geodesic of
$T_i$ from $a_i$ to $b_i$, parametrized by arc length on $[0,d_i]$. Put
$a=(a_1,\dots,a_k)$ and $b=(b_1,\dots,b_k)$. For every permutation $\tau$ of
$\{1,\dots,k\}$, the path which starts at $a$ and traverses
$\alpha_{\tau(1)},\dots,\alpha_{\tau(k)}$ successively, changing one coordinate at a
time, is a geodesic from $a$ to $b$.
\end{lemma}

\begin{proof}
The path has length $S=d_X(a,b)$ by \eqref{eq:distance}.
\end{proof}

A \emph{subtree} of a tree is a connected subgraph; a single vertex is allowed.

\begin{lemma}[Sub-products]\label{lem:subproduct}
Let $T'_i$ be a subtree of $T_i$, $1\le i\le k$, and
$X'=T'_1\times\cdots\times T'_k\subseteq X$.
\begin{enumerate}[label=\textup{(\roman*)},leftmargin=*]
\item $d_{X'}=d_X$ on $X'$; consequently $\Rips(X')\le\Rips(X)$ and
$\four(V(X'))\le\four(V(X))$.
\item If each $T_i$ is the union of an increasing sequence of finite subtrees
$T_i^{(n)}$ and $X^{(n)}=T_1^{(n)}\times\cdots\times T_k^{(n)}$, then
$\Rips(X)=\sup_n\Rips(X^{(n)})$ and $\four(V(X))=\sup_n\four(V(X^{(n)}))$.
\end{enumerate}
\end{lemma}

\begin{proof}
(i) The nearest-point retraction $\rho_i\colon V(T_i)\to V(T'_i)$ maps adjacent
vertices to equal or adjacent vertices. Hence $\rho=(\rho_1,\dots,\rho_k)$ maps every
edge of $X$ onto an edge or a vertex of $X'$; extended affinely to the edges, it is
a $1$-Lipschitz retraction of $X$ onto $X'$. If $x,y\in X'$ and $\gamma$ joins them
in $X$, then $\rho\circ\gamma$ joins them in $X'$ and is not longer; thus
$d_{X'}\le d_X$ on $X'$, and the reverse inequality is obvious. Therefore geodesic
triangles of $X'$ are geodesic triangles of $X$ with the same thinness, and
quadruples of vertices of $X'$ have the same defect in $X'$ and in $X$.

(ii) By (i) both sequences are nondecreasing and bounded by the constants of $X$. A
geodesic triangle of $X$ is compact, hence meets finitely many edges and lies in
$X^{(n)}$ for $n$ large, where by (i) it is a geodesic triangle with the same
thinness; similarly for quadruples of vertices.
\end{proof}

%==========================================================================
\section{The thin-triangle constant}\label{sec:rips}
%==========================================================================

In this section the trees are finite and $k\ge2$; we write $R_j=S-d_j$, so that
$R=\min_jR_j=S-M$, and note that $R_j\ge1$ for every $j$.

\begin{lemma}[Balanced bigon]\label{lem:balancedbigon}
If $M\le R$, then $X$ contains a geodesic bigon with a point at distance exactly
$S/2$ from the other side. Consequently $\Rips(X)\ge S/2$.
\end{lemma}

\begin{proof}
Choose $\alpha_i$, $a$ and $b$ as in \cref{lem:sequential}, let $\gamma^+$ traverse the
coordinates in the order $1,2,\dots,k$ and $\gamma^-$ in the order $k,\dots,1$, and
let $p$ be the midpoint of $\gamma^+$. There are $j$ and $t\in[0,d_j]$ with
\[
 \pi_i(p)=\begin{cases} b_i, & i<j,\\ \alpha_j(t), & i=j,\\ a_i, & i>j,\end{cases}
 \quad P+t=\frac S2=Q+d_j-t,
\]
where $P=\sum_{i<j}d_i$ and $Q=\sum_{i>j}d_i$. A point $q\in\gamma^-$ has, for some
$\ell$ and $s\in[0,d_\ell]$, coordinates $\pi_i(q)=a_i$ for $i<\ell$,
$\pi_\ell(q)=\alpha_\ell(s)$ and $\pi_i(q)=b_i$ for $i>\ell$. Let
$\Sigma=\sum_id_{T_i}(\pi_i(p),\pi_i(q))\le d_X(p,q)$. If $\ell>j$, the coordinates
$i\le j$ give $\Sigma\ge P+t=S/2$; if $\ell<j$, the coordinates $i\ge j$ give
$\Sigma\ge Q+d_j-t=S/2$; if $\ell=j$, then $\Sigma=P+Q+|t-s|\ge S-M=R\ge S/2$. Hence
$d_X(p,\gamma^-)\ge S/2$, with equality since $a\in\gamma^-$ and $d_X(p,a)=S/2$.
\end{proof}

\begin{lemma}[Dominant bigon]\label{lem:dominantbigon}
If $d_1=M\ge R+1$, then $X$ contains a geodesic bigon with a point at distance at
least $R+\frac12$ from the other side. Consequently $\Rips(X)\ge R+\frac12$.
\end{lemma}

\begin{proof}
With the notation of \cref{lem:sequential}, let $\gamma^-$ traverse the coordinates
in the order $1,2,\dots,k$ and $\gamma^+$ in the order $2,\dots,k,1$. Let $p\in\gamma^-$
be the point with $\pi_1(p)=\alpha_1(R+\frac12)$ and $\pi_i(p)=a_i$ for $i\ge2$; it is
the midpoint of an edge $[u,u']$ of $X$ over an edge of $T_1$. Let $q\in\gamma^+$. If
$\pi_1(q)=a_1$, then $d_X(p,q)\ge d_{T_1}(\pi_1(p),a_1)=R+\frac12$. Otherwise
$\pi_i(q)=b_i\ne a_i$ for all $i\ge2$, so $q\notin[u,u']$ and every path from $p$ to
$q$ passes through $u$ or $u'$; since $d_X(u^\circ,q)\ge\sum_{i\ge2}d_{T_i}(a_i,b_i)=R$
for $u^\circ\in\{u,u'\}$, we get $d_X(p,q)\ge R+\frac12$.
\end{proof}

\begin{proposition}[Upper bound]\label{prop:upper}
For every $j$, $\Rips(X)\le R_j+\frac12$.
\end{proposition}

\begin{proof}
Write $\varphi=\pi_j$. We first note: \emph{if $p,w\in X$ do not form an exceptional
pair in the coordinate $j$, then $d_X(p,w)\le d_{T_j}(\varphi(p),\varphi(w))+R_j$.}
Indeed, in \eqref{eq:distance} the coordinates $i\ne j$ contribute at most
$\sum_{i\ne j}d_i=R_j$, also when $p,w$ form an exceptional pair in a coordinate
$m\ne j$, because then the coordinate $m$ contributes
$|s-t|+c=\min\{s+t,2-s-t\}\le1\le d_m$.

Let $\{x,y,z\}$ be a geodesic triangle, $p$ a point of the side $\gamma=[xy]$, and
$\Gamma=[yz]\cup[zx]$, a connected set containing $x$ and $y$.

\emph{Case 1: $x,y$ form an exceptional pair in the coordinate $j$.} Then
$d_X(x,y)\le R_j+1$ by \eqref{eq:distance}, so
$\min\{d_X(p,x),d_X(p,y)\}\le\frac12 d_X(x,y)\le R_j+\frac12$, and $x,y\in\Gamma$.

\emph{Case 2: otherwise.} Let $\sigma=[\varphi(x),\varphi(y)]\subseteq T_j$. By
\cref{lem:geodesics}(i), $\varphi(p)\in\sigma$. Since $\varphi(\Gamma)$ is a connected
subset of $T_j$ containing $\varphi(x)$ and $\varphi(y)$, it contains $\sigma$.

(2a) If $\varphi(p)$ is a vertex, choose $w\in\Gamma$ with $\varphi(w)=\varphi(p)$; then
$d_X(p,w)\le R_j$ by the preliminary remark.

(2b) If $\varphi(p)$ lies in the interior of an edge $e=[t_0,t_1]$ of $T_j$ contained
in $\sigma$, at distance $s$ from $t_0$, choose $w_0,w_1\in\Gamma$ with
$\varphi(w_0)=t_0$ and $\varphi(w_1)=t_1$. By the preliminary remark,
$d_X(p,w_0)\le s+R_j$ and $d_X(p,w_1)\le 1-s+R_j$, and one of them is at most
$R_j+\frac12$.

(2c) If $\varphi(p)$ lies in the interior of an edge $e$ of $T_j$ not contained in
$\sigma$, then $\sigma\cap e$ is a proper subsegment of $e$ containing an interior
point of $e$, so an endpoint of $\sigma$, say $\varphi(x)$, is interior to $e$. If $p$
and $x$ lie on the same edge of $X$, then $d_X(p,x)<1\le R_j$. Otherwise $x,p$ form an
exceptional pair in the coordinate $j$, and by \cref{lem:distance} the subarc of
$\gamma$ from $x$ to $p$ moves the $j$-th coordinate by more than
$d_{T_j}(\varphi(x),\varphi(p))$; this contradicts the monotonicity of
$\varphi\circ\gamma$ given by \cref{lem:geodesics}(i).

In all cases $d_X(p,\Gamma)\le R_j+\frac12$.
\end{proof}

\begin{theorem}[Thin-triangle constant]\label{thm:mainrips}
Let $T_1,\dots,T_k$ be finite nontrivial trees, $k\ge2$. Then
\[
 \Rips(T_1\times\cdots\times T_k)=\min\Bigl\{\frac S2,\ S-M+\frac12\Bigr\}
 =\begin{cases} S/2, & M\le R,\\[1mm] R+1/2, & M>R.\end{cases}
\]
\end{theorem}

\begin{proof}
Since $M$ and $R$ are integers, the two expressions agree: if $M\le R$ then
$S/2\le R<R+\frac12$, and if $M>R$ then $M\ge R+1$ and $R+\frac12\le\frac S2$. The
upper bound $\min\{S/2,R+\frac12\}$ follows from \eqref{eq:diambound},
\cref{cor:diameter} and \cref{prop:upper} with $d_j=M$. The lower bound is
\cref{lem:balancedbigon} if $M\le R$ and \cref{lem:dominantbigon} if $M>R$.
\end{proof}

\begin{corollary}\label{cor:diamsharp}
Under the hypotheses of \cref{thm:mainrips}, $\Rips(X)=\frac12\diam X$ if and only if
$M\le R+1$; if $M\ge R+2$, then $\frac12\diam X-\Rips(X)=\frac{M-R-1}{2}$.
\end{corollary}

\begin{remark}\label{rem:newrips}
For $k=2$, \cref{thm:mainrips} gives $\Rips(T_1\times T_2)=d_1$ if $d_1=d_2$ and
$\min\{d_1,d_2\}+\frac12$ otherwise, in agreement with
\cite{MichelRodriguezSigarretaVilleta2010}; the dominant case also follows from
\cite[Theorem~24]{MichelRodriguezSigarretaVilleta2010}, which states that
$\Rips(T\times G)=\diamV G+\frac12$ for a tree $T$ and a nontrivial graph $G$ with
$\diam T>\diamV G$. For $k\ge3$ the transition is governed by the comparison of the
largest diameter with the \emph{sum} of the others, which is invisible in a
two-factor formula.
\end{remark}

\section{The four-point constant}\label{sec:four}

We work in a more general setting. The \emph{$\ell_1$-product} of metric spaces
$(Y_i,\rho_i)$ is $Y_1\times\cdots\times Y_k$ with the metric
$\rho(x,y)=\sum_i\rho_i(x_i,y_i)$. By \eqref{eq:distance}, $V(X)$ with its graph
metric is the $\ell_1$-product of the vertex sets $V(T_i)$; the same holds for
products of trees with positive edge weights.

\begin{lemma}[Deficit of a tree-like quadruple]\label{lem:deficit}
Let $(Y,\rho)$ be a metric space with $\four(Y)=0$ and $x,y,z,w\in Y$. There are
$\lambda\ge0$, $\varepsilon\in[0,\diam Y]$ and $r\in\{1,2,3\}$ such that
$\sigma_r=\lambda-2\varepsilon$ and the other two sums equal $\lambda$.
\end{lemma}

\begin{proof}
The two largest sums are equal; call them $\lambda$ and write the third as
$\lambda-2\varepsilon$. The three sums lie in $[0,2\diam Y]$, so
$\varepsilon\le\diam Y$.
\end{proof}

\begin{theorem}[Four-point constant of $\ell_1$-products]\label{thm:fourpoint}
Let $Y_1,\dots,Y_k$ be metric spaces with $\four(Y_i)=0$ and finite diameters
$D_1,\dots,D_k$, and let $Y$ be their $\ell_1$-product. Then
$\four(Y)=\cP(D_1,\dots,D_k)$. If every $D_i$ is attained by a pair $a_i,b_i$, the
supremum defining $\four(Y)$ is attained by the quadruple \eqref{eq:quadruple} below.
\end{theorem}

\begin{proof}
\emph{Upper bound.} Fix $x,y,z,w\in Y$. By \cref{lem:deficit} applied in each factor
there are $\lambda_i\ge0$, $\varepsilon_i\in[0,D_i]$ and $r_i\in\{1,2,3\}$ such that
the sum of type $r_i$ in the factor $i$ is $\lambda_i-2\varepsilon_i$ and the other
two are $\lambda_i$. Put $\lambda=\sum_i\lambda_i$ and
$e_r=\sum_{i:r_i=r}\varepsilon_i$. Since the metric is additive,
$\sigma_r=\lambda-2e_r$ for $r=1,2,3$. Relabel the types so that
$e_1\le e_2\le e_3$; then $\four(x,y,z,w)=e_2-e_1\le e_2$. For $E=\{i:r_i=2\}$ we
have $\sum_{i\in E}D_i\ge e_2$, and the complement of $E$ contains the indices of
type $3$, so $\sum_{i\notin E}D_i\ge e_3\ge e_2$. Hence
$\four(x,y,z,w)\le e_2\le\cP(D_1,\dots,D_k)$.

\emph{Lower bound.} Let $E$ attain $\cP$, and let $a_i,b_i\in Y_i$ with
$\rho_i(a_i,b_i)=D'_i$. Define
\begin{equation}\label{eq:quadruple}
 (x_i,y_i,z_i,w_i)=\begin{cases}(a_i,a_i,b_i,b_i), & i\in E,\\
 (a_i,b_i,a_i,b_i), & i\notin E.\end{cases}
\end{equation}
The three sums in the factor $i$ are $(0,2D'_i,2D'_i)$ if $i\in E$ and
$(2D'_i,0,2D'_i)$ otherwise, so in $Y$ they are $2D'_{E^c}$, $2D'_E$ and
$2(D'_E+D'_{E^c})$, with the obvious notation, and
$\four(x,y,z,w)=\min\{D'_E,D'_{E^c}\}$. Letting $D'_i\to D_i$ gives
$\four(Y)\ge\min\{\sum_{i\in E}D_i,\sum_{i\notin E}D_i\}=\cP(D_1,\dots,D_k)$.
\end{proof}

\begin{corollary}[Products of trees]\label{cor:fourtrees}
If $T_1,\dots,T_k$ are finite trees, with unit or with arbitrary positive edge
lengths, then $\four(V(T_1\times\cdots\times T_k))=\cP(d_1,\dots,d_k)$, where $d_i$ is
the diameter of $T_i$.
\end{corollary}

\begin{proof}
Apply \cref{thm:fourpoint} to the vertex sets $V(T_i)$, which are tree metrics.
\end{proof}

\begin{remark}
(1) For products of paths $P_{n_i}$, $d_i=n_i-1$ and \cref{cor:fourtrees} is the
multidimensional grid formula of Coudert and Ducoffe
\cite[Proposition~47]{CoudertDucoffe2016}.
(2) The median-graph theory of \cite{ChepoiDraganEstellonHabibVaxes2008} relates
four-point hyperbolicity to isometrically embedded square grids. The proof above is
more precise for $\ell_1$-products of tree-like spaces: each factor contributes only a
\emph{type} and a deficit bounded by its diameter.
\end{remark}

\begin{corollary}[Diameter-only invariance]\label{cor:diamonly}
If $T_1,\dots,T_k$ and $T'_1,\dots,T'_k$ are finite nontrivial trees, $k\ge2$, with
$\diam T_i=\diam T'_i$ for all $i$, then the two products have the same Rips constant
and the same vertex four-point constant.
\end{corollary}

\begin{proof}
The formulas of \cref{thm:mainrips,cor:fourtrees} depend only on the diameters.
\end{proof}

For instance, the path $P_4$ and the double star $D_{2,2}$ of \cref{fig:branching}
have diameter $3$, so either can replace the other in any product without changing
$\Rips$ or $\four$.

\begin{figure}[ht]
\centering
\begin{tikzpicture}
  \begin{scope}[xshift=0cm]
    \foreach \i in {0,1,2,3}{\node (p\i) at (1.1*\i,0) [circle,fill,inner sep=1.6pt] {};}
    \draw (p0)--(p1)--(p2)--(p3);
    \node at (1.65,-0.6) {$P_4$, $\diam=3$};
  \end{scope}
  \begin{scope}[xshift=6.0cm]
    \node (c0) at (0,0) [circle,fill,inner sep=1.6pt] {};
    \node (c1) at (1.5,0) [circle,fill,inner sep=1.6pt] {};
    \node (a1) at (-0.8,0.7) [circle,fill,inner sep=1.6pt] {};
    \node (a2) at (-0.8,-0.7) [circle,fill,inner sep=1.6pt] {};
    \node (b1) at (2.3,0.7) [circle,fill,inner sep=1.6pt] {};
    \node (b2) at (2.3,-0.7) [circle,fill,inner sep=1.6pt] {};
    \draw (c0)--(c1) (c0)--(a1) (c0)--(a2) (c1)--(b1) (c1)--(b2);
    \node at (0.75,-1.25) {$D_{2,2}$, $\diam=3$};
  \end{scope}
\end{tikzpicture}
\caption{Two trees with the same diameter and different branching.}
\label{fig:branching}
\end{figure}
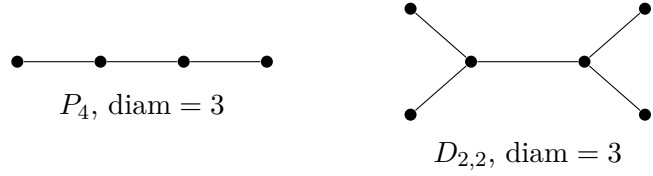

\section{Comparison of the two constants}\label{sec:comparison}

\begin{theorem}[Exact gap]\label{thm:exactgap}
Let $T_1,\dots,T_k$ be finite nontrivial trees, $k\ge2$, and
$\eta=\eta(d_1,\dots,d_k)$. Then $\four(V(X))=\frac{S-\eta}{2}$ and
\begin{equation}\label{eq:gap}
 \Rips(X)-\four(V(X))=\begin{cases}\eta/2, & M\le R,\\[1mm] 1/2, & M>R.\end{cases}
\end{equation}
If $M>R$, then $\eta=M-R$ and $\four(V(X))=R$.
\end{theorem}

\begin{proof}
Since $\min\{q,S-q\}=\frac12(S-|S-2q|)$, \cref{cor:fourtrees} gives
$\four(V(X))=(S-\eta)/2$. If $M\le R$, \cref{thm:mainrips} gives $\Rips(X)=S/2$. If
$M>R$ and $d_j=M$, every $E\ni j$ satisfies
$\sum_{i\in E}d_i-\sum_{i\notin E}d_i\ge M-R>0$, with equality for $E=\{j\}$; hence
$\eta=M-R$, $\four(V(X))=R$ and $\Rips(X)=R+\frac12$.
\end{proof}

\begin{corollary}[Equality criterion]\label{cor:equalitycriteria}
$\Rips(X)=\four(V(X))$ if and only if $\eta=0$, i.e.\ the diameters can be split into
two groups with equal sums.
\end{corollary}

\begin{proof}
If $\eta=0$ then $M\le S/2$, i.e.\ $M\le R$, and the gap is $0$; conversely the gap
vanishes only if $M\le R$ and $\eta=0$.
\end{proof}

\begin{corollary}[Unbounded separation]\label{cor:separation}
For every odd $k\ge3$ and every $D\in\N$, if $T$ is a tree of diameter $2D$, then
$\Rips(T^{\times k})-\four(V(T^{\times k}))=D$.
\end{corollary}

\begin{proof}
All diameters equal $2D$, so $M\le R$ and, $k$ being odd, $\eta=2D$.
\end{proof}

\section{The $\ell_1$-product and the origin of the half unit}\label{sec:l1}

Let $Y_1,\dots,Y_k$ be compact $\R$-trees of diameters $d_1,\dots,d_k$, and let
$Y$ be their $\ell_1$-product. When $Y_i$ is the geodesic realization of a finite tree
$T_i$, $Y$ is the cube complex $|X|_1$ of $X=T_1\times\cdots\times T_k$, in which
every product of edges is filled in by a cube carrying the $\ell_1$-metric, and $X$ is
its $1$-skeleton. Concatenating geodesics of the factors shows that $Y$ is a geodesic
space with $\diam Y=S$.

\begin{lemma}\label{lem:l1geodesics}
Let $\gamma$ be a geodesic of $Y$ from $x$ to $y$. Then every projection
$\pi_i\circ\gamma$ traverses the geodesic $[\pi_i(x),\pi_i(y)]$ of $Y_i$
monotonically.
\end{lemma}

\begin{proof}
For a curve in an $\ell_1$-product, taking common refinements of partitions shows
that $L(\gamma)=\sum_iL(\pi_i\circ\gamma)$. Since
$L(\gamma)=\rho(x,y)=\sum_i\rho_i(\pi_i(x),\pi_i(y))$ and
$L(\pi_i\circ\gamma)\ge\rho_i(\pi_i(x),\pi_i(y))$, all these inequalities are
equalities, and a path in an $\R$-tree of length equal to the distance between its
endpoints is a monotone parametrization of the geodesic.
\end{proof}

\begin{theorem}\label{thm:l1}
Let $Y$ be the $\ell_1$-product of compact $\R$-trees $Y_1,\dots,Y_k$ of positive
diameters $d_1,\dots,d_k$, $k\ge2$. Then
\[
 \Rips(Y)=\min\Bigl\{\frac S2,\ S-M\Bigr\},\quad \four(Y)=\cP(d_1,\dots,d_k).
\]
Consequently $\Rips(Y)-\four(Y)$ equals $\eta/2$ if $M\le R$ and $0$ if $M>R$.
\end{theorem}

\begin{proof}
The four-point formula is \cref{thm:fourpoint}.

\emph{Upper bound.} $\Rips(Y)\le S/2$ by \eqref{eq:diambound}. Fix $j$; let
$\{x,y,z\}$ be a geodesic triangle, $p\in[xy]$ and $\Gamma=[yz]\cup[zx]$. By
\cref{lem:l1geodesics}, $\pi_j(p)\in[\pi_j(x),\pi_j(y)]$, and this geodesic is
contained in the connected set $\pi_j(\Gamma)$. Choose $w\in\Gamma$ with
$\pi_j(w)=\pi_j(p)$; then $\rho(p,w)=\sum_{i\ne j}\rho_i(\pi_i(p),\pi_i(w))\le S-d_j$.
Hence $\Rips(Y)\le S-M$.

\emph{Lower bound.} If $M\le R$, the proof of \cref{lem:balancedbigon} applies
verbatim, the projection bound being now an equality, and gives $\Rips(Y)\ge S/2$. If
$d_1=M>R$, let $\gamma^-$ and $\gamma^+$ be as in the proof of
\cref{lem:dominantbigon}, and let $p\in\gamma^-$ have $\pi_1(p)=\alpha_1(\tau)$ with
$R<\tau\le M$ and $\pi_i(p)=a_i$ for $i\ge2$. For $q\in\gamma^+$, either
$\pi_1(q)=a_1$ and $\rho(p,q)\ge\tau>R$, or $\pi_i(q)=b_i$ for all $i\ge2$ and
$\rho(p,q)\ge R$. Hence $\Rips(Y)\ge R=S-M$. The two lower bounds match the upper bound
$\min\{S/2,S-M\}$ in the respective regimes.
\end{proof}

\begin{corollary}\label{cor:halfunit}
For a product $X$ of finite nontrivial trees, $k\ge2$,
\[
 \Rips(X)-\Rips(|X|_1)=\begin{cases}0, & M\le R,\\ 1/2, & M>R,\end{cases}
 \quad \four(V(X))=\four(|X|_1).
\]
\end{corollary}

Thus the two conventions differ, in the dominant regime, only because the cubes of
$|X|_1$ are not filled in $X$: the geodesics of $X$ cannot cut through the squares,
and the extremal point of \cref{lem:dominantbigon} is the midpoint of an edge. In the
balanced regime the discrepancy $\eta/2$ is of a different, arithmetic nature, and it
persists in $|X|_1$.

\section{Consequences}\label{sec:families}

\subsection{Powers, grids and hypercubes}

\begin{corollary}\label{cor:powers}
Let $T$ be a finite nontrivial tree with $d=\diam T$ and $k\ge2$. Then
\[
 \Rips(T^{\times k})=\frac{kd}{2},\quad
 \four(V(T^{\times k}))=d\Bigl\lfloor\frac k2\Bigr\rfloor .
\]
In particular, for the grid $P_s^{\times k}$, $s\ge2$, the two constants are
$\frac{k(s-1)}2$ and $(s-1)\lfloor\frac k2\rfloor$, and for the hypercube
$Q_k=P_2^{\times k}$ they are $\frac k2$ and $\lfloor\frac k2\rfloor$. For a product of
paths $P_{n_1}\times\cdots\times P_{n_k}$, Theorems A and B apply with
$d_i=n_i-1$.
\end{corollary}

\begin{proof}
All diameters equal $d$, so $M=d\le(k-1)d=R$, and the largest subset sum not exceeding
$kd/2$ is $d\lfloor k/2\rfloor$.
\end{proof}

\subsection{Unbounded factors}

\begin{theorem}\label{thm:unbounded}
Let $T_1,\dots,T_k$ be nontrivial locally finite trees, $k\ge2$.
\begin{enumerate}[label=\textup{(\alph*)},leftmargin=*]
\item If at least two factors are unbounded, then $\Rips(X)=\four(V(X))=\infty$.
\item If exactly one factor is unbounded and $R$ is the sum of the diameters of the
others, then $\Rips(X)=R+\frac12$ and $\four(V(X))=R$.
\end{enumerate}
In particular $X$ is hyperbolic, for either convention, if and only if at most one
factor is unbounded.
\end{theorem}

\begin{proof}
(a) If $T_1,T_2$ are unbounded, choose geodesic paths of length $n$ in them and single
vertices in the other factors. The sub-product is isomorphic to
$P_{n+1}\times P_{n+1}$, whose constants are both equal to $n$ by
\cref{cor:powers}; apply \cref{lem:subproduct}(i).

(b) Let $T_1$ be unbounded and $o\in V(T_1)$. The balls $T_1^{(n)}$ of radius $n$ about
$o$ are finite subtrees with union $T_1$ and $\diam T_1^{(n)}\ge n$. For $n>R$ the
product $T_1^{(n)}\times T_2\times\cdots\times T_k$ is in the dominant regime, so its
constants are $R+\frac12$ and $R$ by \cref{thm:mainrips,thm:exactgap}; apply
\cref{lem:subproduct}(ii).
\end{proof}

The non-hyperbolicity in case (a) also follows from the characterization of
hyperbolic Cartesian products in \cite{MichelRodriguezSigarretaVilleta2010}.

\subsection{Algorithmic consequences}\label{sec:algorithms}

A direct computation on $X$ is prohibitive, since $|V(X)|=\prod_i|V(T_i)|$ and the
general algorithm of \cite{FournierIsmailVigneron2015} for the four-point constant of
an $n$-point space runs in time $O(n^{3.69})$.

\begin{proposition}\label{prop:algorithm}
Let $T_1,\dots,T_k$ be finite nontrivial trees given by adjacency lists, $k\ge2$, and
$N=\sum_i|V(T_i)|$. Then $\Rips(X)$ can be computed in time $O(N)$, and
$\four(V(X))$, together with an extremal quadruple, in time $O(kN)$ and space
$O(N)$. If instead the diameters are given in binary, deciding whether
$\four(V(X))=S/2$ is NP-complete in the weak sense.
\end{proposition}

\begin{proof}
Each diameter is found in linear time by two breadth-first searches, together with a
diametral pair; $\Rips(X)$ then follows from \cref{thm:mainrips}. By
\cref{cor:fourtrees}, $\four(V(X))$ is the largest subset sum of $d_1,\dots,d_k$ not
exceeding $S/2$, computed by the standard dynamic program in time $O(kS)$ and space
$O(S)$, with $S<N$; backtracking gives an optimal set $E$, and \eqref{eq:quadruple}
gives an extremal quadruple. Finally, $\four(V(X))=S/2$ if and only if
$d_1,\dots,d_k$ admit a perfect partition, which is the \textsc{Partition} problem
\cite{GareyJohnson1979}.
\end{proof}

For example, $P_{101}^{\times5}$ has $101^5=10\,510\,100\,501$ vertices, its factors
have $505$ vertices in total, and \cref{cor:powers} gives at once $\Rips=250$ and
$\four=200$. Extremal bigons are obtained in the same way from diametral paths of the
factors (\cref{lem:balancedbigon,lem:dominantbigon}).

\section{Examples and computational checks}\label{sec:numerics}

The results above are proved; the computations reported here are independent checks.
In \cref{tab:four}, the last column was obtained by evaluating the four-point defect
on all quadruples of vertices. In \cref{tab:rips}, the last column was obtained by
enumerating all geodesic triangles with vertices in the set of vertices and edge
midpoints, with all possible choices of the three geodesics, in the graph obtained by
subdividing each edge into two (and also into four for $P_3\times P_2$ and
$P_4\times P_2$), and computing the
largest distance from a point of one side to the union of the other two. Here
$K_{1,3}$ is the star with three leaves and $D_{2,2}$ the double star of
\cref{fig:branching}.

\begin{table}[ht]
\centering
\small
\caption{Four-point constants; $\eta$ is the partition imbalance.}
\label{tab:four}
\begin{tabular}{@{}lccccccc@{}}
\toprule
$X$ & $(d_i)$ & $|V(X)|$ & $\eta$ & $\Rips$ & $\four$ (formula) & gap &
$\four$ (exhaustive) \\
\midrule
$Q_3=P_2^{\times3}$ & $(1,1,1)$ & $8$ & $1$ & $3/2$ & $1$ & $1/2$ & $1$ \\
$P_3^{\times3}$ & $(2,2,2)$ & $27$ & $2$ & $3$ & $2$ & $1$ & $2$ \\
$P_5\times P_4\times P_4$ & $(4,3,3)$ & $80$ & $2$ & $5$ & $4$ & $1$ & $4$ \\
$P_9\times P_3\times P_2$ & $(8,2,1)$ & $54$ & $5$ & $7/2$ & $3$ & $1/2$ & $3$ \\
$P_4\times P_3\times P_3$ & $(3,2,2)$ & $36$ & $1$ & $7/2$ & $3$ & $1/2$ & $3$ \\
$D_{2,2}\times K_{1,3}\times P_3$ & $(3,2,2)$ & $72$ & $1$ & $7/2$ & $3$ & $1/2$ & $3$ \\
\bottomrule
\end{tabular}
\end{table}

\begin{table}[ht]
\centering
\small
\caption{Thin-triangle constants: formula of \cref{thm:mainrips} and exhaustive
computation.}
\label{tab:rips}
\begin{tabular}{@{}lccccc@{}}
\toprule
$X$ & $(d_i)$ & regime & $\Rips$ (formula) & $\Rips$ (exhaustive) \\
\midrule
$P_3\times P_2$ & $(2,1)$ & $M>R$ & $3/2$ & $3/2$ \\
$P_4\times P_2$ & $(3,1)$ & $M>R$ & $3/2$ & $3/2$ \\
$P_5\times P_2$ & $(4,1)$ & $M>R$ & $3/2$ & $3/2$ \\
$K_{1,3}\times P_2$ & $(2,1)$ & $M>R$ & $3/2$ & $3/2$ \\
$P_3\times P_3$ & $(2,2)$ & $M\le R$ & $2$ & $2$ \\
$Q_3$ & $(1,1,1)$ & $M\le R$ & $3/2$ & $3/2$ \\
$P_3\times P_2\times P_2$ & $(2,1,1)$ & $M\le R$ & $2$ & $2$ \\
$P_4\times P_2\times P_2$ & $(3,1,1)$ & $M=R+1$ & $5/2$ & $5/2$ \\
$P_5\times P_2\times P_2$ & $(4,1,1)$ & $M\ge R+2$ & $5/2$ & $5/2$ \\
\bottomrule
\end{tabular}
\end{table}

\begin{example}[Diameter vector $(4,3,3)$]\label{ex:433}
Let $X=P_5\times P_4\times P_4$, so $S=10$, $M=4\le R=6$ and $\Rips(X)=5$. The subset
sums of $(4,3,3)$ are $0,3,4,6,7,10$, so $\four(V(X))=4$, $\eta=2$ and the gap is
$1$. With $E=\{1\}$, the quadruple \eqref{eq:quadruple} has sums $(12,8,20)$ and
defect $4$. The geodesics $\gamma^+$ and $\gamma^-$ of \cref{lem:balancedbigon}
traverse the coordinates in the orders $(1,2,3)$ and $(3,2,1)$
(\cref{fig:balancedbox}); the midpoint of $\gamma^+$ is $p=(4,1,0)$ and
$d(p,\gamma^-)=5=\frac12\diam X$.
\end{example}

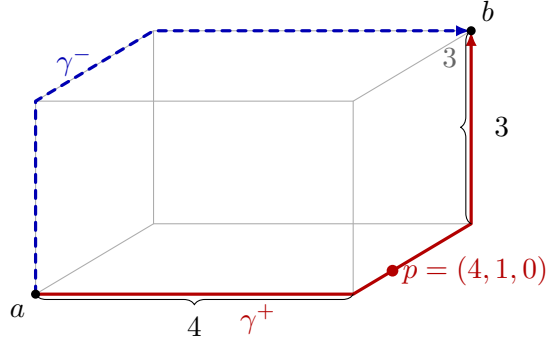
\begin{figure}[ht]
\centering
\begin{tikzpicture}[x={(1.05cm,0cm)},y={(0.52cm,0.31cm)},z={(0cm,0.85cm)},line cap=round,line join=round]
  \coordinate (A) at (0,0,0);
  \coordinate (X) at (4,0,0);
  \coordinate (Y) at (0,3,0);
  \coordinate (XY) at (4,3,0);
  \coordinate (Z) at (0,0,3);
  \coordinate (XZ) at (4,0,3);
  \coordinate (YZ) at (0,3,3);
  \coordinate (B) at (4,3,3);
  \coordinate (P) at (4,1,0);
  \draw[gray!65] (A)--(X)--(XY)--(Y)--cycle;
  \draw[gray!65] (Z)--(XZ)--(B)--(YZ)--cycle;
  \draw[gray!65] (A)--(Z) (X)--(XZ) (Y)--(YZ) (XY)--(B);
  \draw[very thick,red!70!black,-{Latex[length=2mm]}] (A)--(X)--(XY)--(B);
  \draw[very thick,blue!70!black,dashed,-{Latex[length=2mm]}] (A)--(Z)--(YZ)--(B);
  \fill (A) circle (1.7pt) node[below left] {$a$};
  \fill (B) circle (1.7pt) node[above right] {$b$};
  \fill[red!70!black] (P) circle (2.2pt) node[right] {$p=(4,1,0)$};
  \node[red!70!black,below] at (2.8,0,0) {$\gamma^+$};
  \node[blue!70!black,left] at (0,1.7,3) {$\gamma^-$};
  \draw[decorate,decoration={brace,mirror,amplitude=4pt}] (A)--(X) node[midway,below=5pt] {$4$};
  \node[gray!80!black] at ($(XZ)!0.5!(B)+(0.30,0.35)$) {$3$};
  \draw[decorate,decoration={brace,amplitude=4pt}] (XY)--(B) node[midway,right=5pt] {$3$};
\end{tikzpicture}
\caption{The balanced bigon for the diameter vector $(4,3,3)$ (schematic: the box
represents the chosen diametral paths of $P_5\times P_4\times P_4$). The marked
midpoint satisfies $d(p,\gamma^-)=5=S/2$.}
\label{fig:balancedbox}
\end{figure}

\begin{example}
The vector $(5,5,2,2)$ has $\eta=0$, so $\Rips=\four=7$: the two conventions can agree
for products of unequal factors, equality being controlled by the arithmetic of the
diameters and not by symmetry. For $P_9\times P_3\times P_2$ we have $M=8>R=3$, so
$\Rips=\frac72$, $\four=3$ and, for the $\ell_1$ cube complex, $\Rips(|X|_1)=3$; the
point of $\gamma^-$ with coordinates $(3.5,0,0)$ is at distance $\frac72$ from
$\gamma^+$ in $X$.
\end{example}

\section{Conclusions and outlook}

We have obtained a sharp and unified description of Gromov hyperbolicity for Cartesian products of finitely many trees. A central feature of the theory is that the two natural hyperbolicity constants are completely determined by the diameter vector of the factors, although through fundamentally different mechanisms. Writing
\[
S=\sum_{i=1}^k d_i,
\quad
M=\max_{1\leq i\leq k} d_i,
\]
the thin-triangle constant of the geodesic realization is
\[
\Rips(T_1\times\cdots\times T_k)
=
\min\left\{
\frac{S}{2},
S-M+\frac12
\right\}.
\]
Thus, for $k\geq3$, the geometry exhibits a genuine many-factor balancing phenomenon: the shorter factors may collectively balance the largest one and force the universal bound $\Rips(X)=\frac12\diam X$ to be attained. When one factor dominates, the remaining factors determine the constant, together with an exact half-edge correction coming from the $1$-skeleton geometry.

For the vertex metric, the sharp four-point constant is governed instead by an arithmetic optimization problem:
\[
\four(V(X))
=
\max_{E\subseteq[k]}
\min\left\{
\sum_{i\in E}d_i,
\sum_{i\notin E}d_i
\right\}.
\]
This extends the known multidimensional-grid formula from path factors to arbitrary trees and shows that branching is invisible to both sharp constants. Consequently, trees with the same diameter can be substituted for one another in any factor without changing either hyperbolicity invariant.

The exact comparison between the two conventions reveals a further structural phenomenon. If $\eta$ denotes the minimum imbalance of a two-way partition of the diameter vector, then
\[
\Rips(X)-\four(V(X))
=
\begin{cases}
\eta/2, & M\leq S-M,\\[1mm]
1/2, & M>S-M.
\end{cases}
\]
Hence the discrepancy between thin-triangle and four-point hyperbolicity is controlled either by the arithmetic of the diameter partition or by a purely geometric half-edge effect. Passing from the graph $X$ to the associated $\ell_1$ cube complex removes exactly this half-unit contribution:
\[
\Rips(|X|_1)
=
\min\left\{
\frac{S}{2},
S-M
\right\},
\quad
\four(|X|_1)
=
\four(V(X)).
\]
This identifies precisely which part of the discrepancy is caused by the discrete $1$-skeleton and which part is intrinsic to the diameter vector.

Beyond the finite case, the same framework yields a complete hyperbolicity criterion for locally finite products: the product is hyperbolic precisely when at most one factor is unbounded. The sharp formulas also lead to efficient algorithms that operate on the factors themselves rather than on the exponentially larger Cartesian product. In particular, $\Rips(X)$ is determined by the two scalar quantities $(S,M)$, while the computation of $\four(V(X))$ reduces exactly to a classical subset-sum problem on the factor diameters.

Taken together, these results provide a sharp geometric, arithmetic, and algorithmic description of hyperbolicity in products of trees. They show that the complexity of the full Cartesian product collapses to a small collection of diameter data, while still preserving a nontrivial distinction between geodesic and four-point hyperbolicity. This combination of exact formulas, extremal constructions, metric interpretation, and computational consequences provides a natural basis for extending the theory to weighted tree products, more general median structures, and products involving non-tree hyperbolic factors.

\section*{Declarations}

\subsection*{Ethical Approval}
Not applicable.

\subsection*{Consent to Participate}
Not applicable.

\subsection*{Consent to Publish}
Not applicable.

\subsection*{Data Availability Statement}
The experimental archive supporting the findings of this study, comprising the
per-run logs of the 50 independent executions per instance, the recorded budget
traces, the paired Wilcoxon comparisons, the archived best tours, and the tuning
driver, is available from the corresponding author upon reasonable request. All
benchmark instances are taken from the public TSPLIB and Waterloo National TSP
libraries cited in the text.

\subsection*{Author Contributions}
Conceptualization, J.-M. and A.S.-L.; methodology, J.-M. and A.S.-L.; software,  J.-M., A.S.-L. and JM. -V.; validation, J.M., A.S.-L., and JM. -V.; formal analysis, J.-M., A.S.-L.; investigation,J.-M. and A.S.-L. and JM. -V.; resources, J.-M. and A.S.-L.; data curation, J.-M., A.S.-L. and JM. -V.; writing--original draft preparation, J.-M., A.S.-L. and JM. -V.; writing--review and editing, J.-M.; visualization, A.S.-L. and J.M; supervision, J.-M., A.S.-L. and JM. -V.; project administration, J.-M. and A.S.-L. All authors have read and agreed to the published version of the manuscript.

\subsection*{Funding}
This research received no external funding.

\subsection*{Competing Interests}
The authors declare that they have no competing interests.

\end{document}